\documentclass[11pt,reqno]{amsart}

\usepackage[margin=1.15in]{geometry}
\usepackage{amsmath,amssymb,amsthm,mathtools}
\usepackage{booktabs}
\usepackage{xcolor}
\usepackage[colorlinks,linkcolor=blue!55!black,citecolor=blue!55!black,
            urlcolor=blue!55!black]{hyperref}
\usepackage{tikz}
\definecolor{covhue}{RGB}{28,74,116}
\colorlet{covtext}{covhue!60!black}
\usepackage{microtype}

\theoremstyle{plain}
\newtheorem{theorem}{Theorem}[section]
\newtheorem{lemma}[theorem]{Lemma}
\newtheorem{proposition}[theorem]{Proposition}
\newtheorem{corollary}[theorem]{Corollary}
\theoremstyle{definition}
\newtheorem{definition}[theorem]{Definition}

\DeclareMathOperator{\Real}{Re}
\DeclareMathOperator{\Imag}{Im}
\DeclareMathOperator{\artanh}{artanh}
\newcommand{\D}{\mathbb D}
\newcommand{\C}{\mathbb C}
\newcommand{\R}{\mathbb R}
\newcommand{\Binf}{B_{\infty}}
\newcommand{\abs}[1]{\lvert#1\rvert}

\begin{document}

\title{A computer-assisted lower bound for Landau's constant}
\author{Frank Wikstr\"om}
\address{Centre for Mathematical Sciences, Lund University,
Box 118, SE-221 00 Lund, Sweden}
\email{frank.wikstrom@math.lth.se}

\begin{abstract}
We prove the lower bounds
\[
   \Binf>0.51,\qquad L>0.51,
\]
where $\Binf$ is the locally univalent Bloch constant and $L$ is Landau's
constant, improving the bound $\frac{1}{2}+2\cdot10^{-8}$ of Chen and Shiba.  The
argument passes to the logarithm $g=\log f'$ of a normalized locally univalent
Bloch function, where the Bloch condition becomes a one-sided linear constraint
on $\Real g$; positive Toeplitz moment matrices then provide a
finite-dimensional outer relaxation for the low logarithmic coefficients, and
any non-negative dual vector for a linear program over that relaxation certifies
a bound.  Floating-point linear programming is used only to discover such dual
vectors; the same duality also certifies the second-derivative bounds that
control the Taylor allowances and the coefficient caps that contract the cover.
An independent arbitrary-precision interval-arithmetic program reconstructs
every witness and checks a finite cover of the normalized coefficient body by
$1179$ boxes, the smallest certified radius being $0.51003$.  All numerical
premises were verified with Arb at $80$ decimal digits.
\end{abstract}


\maketitle

\section{Introduction}

Let $f$ be holomorphic on the unit disc $\D$.  For $a\in\D$, let
$r(a,f)$ be the radius of the largest disc on the Riemann surface of $f$
centered at $f(a)$ on which an inverse branch of $f$ is defined.  Put
\[
  r(f)=\sup_{a\in\D}r(a,f).
\]
The locally univalent Bloch constant is
\[
  \Binf=\inf r(f),
\]
where the infimum is taken over locally univalent functions with $f'(0)=1$.
Similarly, for a holomorphic function $f$ on $\D$, let $\lambda(f)$ be the supremum
of the radii of Euclidean discs contained in $f(\D)$.  Landau's constant is
\[
  L=\inf\{\lambda(f):f\text{ holomorphic on $\D$}, f'(0)=1\}.
\]
Both constants can be computed on the normalized Bloch class
\begin{equation}\label{eq:normalized-class}
 f(0)=0,\qquad f'(0)=1,\qquad
 (1-\abs{z}^2)\abs{f'(z)} \leq  1\quad(z\in\D):
\end{equation}
the infimum defining $\Binf$ is unchanged if it is taken over the locally
univalent functions satisfying~\eqref{eq:normalized-class} only, and
\begin{equation}\label{eq:relation}
 \Binf \leq  L.
\end{equation}
Both facts are standard (Landau~\cite{Landau}, Pommerenke~\cite{Pommerenke},
Minda~\cite{Minda}, Chen--Shiba~\cite{ChenShiba}).  Since
\eqref{eq:relation} is what transfers the main result from $\Binf$ to $L$,
and since neither fact asserts that the functions in Landau's problem are
locally univalent, we include the short proofs as
Lemma~\ref{lem:reduction}.

For Landau's constant the lower bound $L\geq1/2$ goes back to
Ahlfors~\cite{Ahlfors}, whose extension of the Schwarz lemma also gave
$B\geq\sqrt3/4$.  The classical lower bound $\Binf\geq  1/2$ is due to Peschl~\cite{Peschl};
Pommerenke~\cite{Pommerenke} proved that the inequality is strict, and that also
follows from the sharp distortion theorem of Liu and Minda~\cite{LiuMinda} for
locally univalent Bloch functions. Yanagihara~\cite{Yanagihara} proved
$\Binf>1/2+10^{-417}$, and Chen and Shiba~\cite{ChenShiba} improved this bound
to
\[
  \Binf>\frac{1}{2}+2\cdot10^{-8}
\]
by refining the escaping-path geometry.

From above, Rademacher~\cite{Rademacher} proved
\begin{equation}\label{eq:upper}
  L\leq\frac{\Gamma(1/3)\Gamma(5/6)}{\Gamma(1/6)}=0.5432589653\ldots,
\end{equation}
and conjectured that this value is $L$ itself; by~\eqref{eq:relation} the
same number bounds $\Binf$ from above.  (For Bloch's constant the analogous
upper bound is that of Ahlfors and Grunsky~\cite{AhlforsGrunsky}, likewise
conjectured to be sharp.)  Measured against~\eqref{eq:upper}, the interval
left open by Chen and Shiba has length $0.0433$, of which
Theorem~\ref{thm:main} below removes $0.01$, a little under a quarter,
leaving
\[
  0.51<\Binf\leq L\leq0.5433 .
\]
In a different direction, Rettinger~\cite{Rettinger} proved that $L$ is a
computable real number by giving an algorithm that approximates it to any
prescribed precision.  That result concerns the existence of such an
algorithm; the present paper is concerned with an explicit bound.

Our approach keeps the escaping-path geometry of Chen and Shiba but replaces
the pointwise estimates on which it rests by certified solutions of linear
programs.  The starting point is the logarithm of the derivative, which local
univalence makes globally available and which Chen and Shiba already use in
their coefficient estimates: we write
\[
  f'(z)=\exp g(z),\qquad\text{with }
  g(z)=c_2z^2+c_3z^3+\cdots,
\]
and the Bloch constraint becomes
\begin{equation}\label{eq:log-body-intro}
 \Real g(z) \leq  q(\abs{z}),\qquad \text{where } q(r)=-\log(1-r^2).
\end{equation}
This is a semi-infinite family of linear inequalities in the Taylor
coefficients of~$g$, and every lower bound we need for $\Real g$ along an
escaping path is a lower bound for a linear functional of those coefficients
over the set that the inequalities cut out.  Minimizing such a functional is
an infinite-dimensional linear program, and every finite non-negative
combination of the constraints is a rigorous affine lower bound for its
value, whatever method produced the multipliers.

This certificate scheme was introduced in~\cite{Wikstrom} for Bloch's constant
$B$, where it produced an improved lower bound $B\geq \sqrt3/4+0.0153$. In that
setting the constraint set was the class $(1-\abs{z}^2)\abs{f(z)} \leq 1$ itself
and the functional a radial integral. Here, the class is described through its
logarithm, which is what makes the constraint linear in the first place, and the
finite relaxations come from a different source: since $q(s)-\Real
g(se^{i\theta})$ is a non-negative trigonometric series for every $s$, its
Toeplitz moment matrices are positive semidefinite, and testing them against
arbitrary vectors produces necessary linear inequalities without any coefficient
tail.  Floating-point linear programming is used only to discover non-negative
dual vectors; an independent interval arithmetic computation encloses the
residual of each dual inequality, and thereby the bound it certifies.

Two further applications of this duality account for most of the final gain.
The escaping-path argument needs lower bounds for $\Real g$ on arcs and annuli,
not only at the finitely many points at which duals are stored, and the Taylor
allowance between stored points is governed by an upper bound for the second
derivative of the truncated logarithm. Replacing the universal coefficient
bounds in that allowance by a certified one-sided bound, obtained from a dual
vector for the second-derivative functional (Lemma~\ref{lem:curvature}), removes
what was the dominant loss. Applied to the coordinate functionals themselves,
duality certifies caps on single coefficients that contract the region under
consideration beyond the quadratic rigidity of Lemma~\ref{lem:quadratic}
(Lemma~\ref{lem:cap}).

The proof does not sample the class of functions.  It covers the complete
normalized coefficient body by $1179$ closed boxes in seven real variables: the
real and imaginary parts of $c_2,\ldots,c_5$ (where $c_2$ may be taken real and
non-negative).  All coefficients from $c_6$ onward remain free.  On each box,
several independently verified directional profiles are combined, including
annulus-wise maxima of simultaneous estimates on identical radial grids, and the
verifier checks the cover, every dual inequality, every curvature bound and cap,
all angular and radial gaps, and every omitted Taylor coefficient using Arb
interval arithmetic~\cite{Arb}.  The main result is as follows.

\begin{theorem}\label{thm:main}
The locally univalent Bloch constant and Landau's constant satisfy
\[
  \Binf>0.51,\qquad L>0.51.
\]
More precisely, the computer-assisted part of the proof gives a common lower
bound, valid for every locally univalent member of the normalized class, in the
enclosure
\[
  [0.5100305329006052049176652\mathbin{\pm}4.18\cdot10^{-26}].
\]
\end{theorem}

\subsection*{Plan of the proof}

The argument has five steps, and each may be read on its own once the
statement of the preceding one is accepted.

\begin{enumerate}
\item[(i)] \emph{Linearization} (Section~\ref{sec:log}).  Local univalence
makes $g=\log f'$ a single-valued holomorphic function, and the Bloch
condition~\eqref{eq:normalized-class} turns into the one-sided
inequality~\eqref{eq:log-body-intro} for $\Real g$.  This is a family of
\emph{linear} inequalities in the Taylor coefficients~$c_k$ of~$g$, indexed
by the point $z\in\D$.  Normalization reduces the coefficient body to
\eqref{eq:normalization}, and Lemma~\ref{lem:coeff} bounds every~$c_k$.

\item[(ii)] \emph{Finite relaxation} (Section~\ref{sec:relax}).  Keeping the
coefficients $c_2,\ldots,c_m$ as unknowns and discarding the rest, we record
finitely many linear consequences of~\eqref{eq:log-body-intro}: the moment
cuts~\eqref{eq:moment-cut}, which come from the positive semi-definiteness
of the Toeplitz matrices of the non-negative function
$q(s)-\Real g(se^{i\theta})$ and carry no coefficient tail, the point
cuts~\eqref{eq:point-cut}, and the tangent cuts~\eqref{eq:tangent}.  Write
them as $Ax \leq  b$.  Every set that these inequalities cut out contains the
true coefficient body, so a lower bound over the former is a lower bound
over the latter.

\item[(iii)] \emph{Certified bounds by duality} (Lemma~\ref{lem:dual}).  For
a linear objective~$d$ and a coefficient box $l \leq  x \leq  h$, any
non-negative vector~$y$ produces the rigorous lower bound~\eqref{eq:dual}
for~$d^Tx$.  Nothing is assumed about how~$y$ was found and no optimality is
needed; only the residual $d+A^Ty$ has to be enclosed.  This is the
mechanism behind the separation of search from proof just described.

\item[(iv)] \emph{From coefficients to a schlicht disc}
(Section~\ref{sec:paths}).  The Liu--Minda estimate~\eqref{eq:LM} integrated
along the lift of an escaping radial ray gives the baseline radius~$1/2$
(Lemma~\ref{lem:escape}), and any better lower bound for $\abs{f'}$ on part of
the radial range adds its excess (Corollary~\ref{cor:excess}). Step~(iii),
applied to the objectives $\Real\sum c_kz^k$ at finitely many points~$z$, gives
lower bounds for $\Real g$ that exceed the baseline; Lemma~\ref{lem:curvature},
fed by duals for the second-derivative objectives, propagates them from those
points to whole arcs and annuli. A Schwarz--Pick estimate converts source angles
into image angles; and a shift of the image center (Section~\ref{sec:shift})
balances the remaining directions.  The outcome is
Proposition~\ref{prop:conditional}: a certified radius for every function whose
low coefficients lie in a given box.

\item[(v)] \emph{Covering the coefficient body} (Section~\ref{sec:cover}).
The seven real coordinates~\eqref{eq:seven} are covered by~$1179$ closed
boxes, contracted where possible by Lemma~\ref{lem:quadratic} and
Lemma~\ref{lem:cap}; all coefficients from~$c_6$ on remain free throughout.
Step~(iv) on each box, and the minimum over the cover, prove
Theorem~\ref{thm:main}.
\end{enumerate}

Section~\ref{sec:verify} describes the independent verifier that reconstructs
every inequality used in steps~(ii)--(v) in interval arithmetic.

\section{Logarithmic normalization and coefficient bounds}
\label{sec:log}

For convenience of the reader, we describe the standard reduction announced in
the introduction.  Recall that $r(a,f)\geq\rho$ means that some inverse branch
of $f$ is defined on the disc of radius $\rho$ about $f(a)$ and takes the value
$a$ at its centre; such a disc is contained in $f(\D)$.

\begin{lemma}\label{lem:reduction} \textup{(a)} $\Binf=\inf r(f)$, where the
infimum is taken over the locally univalent functions
satisfying~\eqref{eq:normalized-class}. \textup{(b)} $\Binf \leq  L$.
\end{lemma}

\begin{proof}
(a) The functions in (a) form a subclass of those in the definition of
$\Binf$, so the infimum in (a) is at least $\Binf$.  Conversely, let $f$ be
locally univalent on $\D$ with $f'(0)=1$, fix $0<r<1$, and put
$f_r(z)=f(rz)$.  Since $f_r$ is holomorphic on a neighbourhood of the closed
disc, $(1-\abs{z}^2)\abs{f_r'(z)}$ is continuous on the closed disc,
vanishes on the circle and attains its maximum $M_r\geq\abs{f_r'(0)}=r$ at
some $z_0\in\D$.  Let $T$ be an automorphism of $\D$ with $T(0)=z_0$ and put
\[
  F=A\circ f_r\circ T,\qquad
  A(w)=\frac{w-f_r(z_0)}{(f_r\circ T)'(0)},\qquad
  \abs{(f_r\circ T)'(0)}=(1-\abs{z_0}^2)\abs{f_r'(z_0)}=M_r .
\]
Then $F$ is locally univalent with $F(0)=0$ and $F'(0)=1$, and the identity
$(1-\abs{z}^2)\abs{T'(z)}=1-\abs{T(z)}^2$ gives
\[
  (1-\abs{z}^2)\abs{F'(z)}
   =\frac{(1-\abs{T(z)}^2)\abs{f_r'(T(z))}}{M_r} \leq  1,
\]
so $F$ satisfies~\eqref{eq:normalized-class}.  If $\psi$ is an inverse
branch of $F$ on a disc of radius $\rho$ about $F(a)$, then
$T\circ\psi\circ A$ is an inverse branch of $f_r$ on the disc of radius
$M_r\rho$ about $f_r(T(a))$; and if $\psi_r$ is an inverse branch of $f_r$
on a disc about $f_r(b)$, then $r\psi_r$ is an inverse branch of $f$ on the
same disc, about $f(rb)$.  Hence
\[
  r(f)\geq r(f_r)\geq M_r\,r(F)\geq r\,r(F)\geq r\,C,
\]
where $C$ is the infimum in (a), and $r\to1$ gives $r(f)\geq C$.

(b) Let $f$ be holomorphic on $\D$ with $f'(0)=1$ and put $\Omega=f(\D)$.
If $\C\setminus\Omega$ has at most one point, $\Omega$ contains discs of
every radius and $\lambda(f)=\infty$.  Otherwise $\Omega$ has a holomorphic
universal covering map $F:\D\to\Omega$ with $F(0)=f(0)$, and $f$ lifts to a
holomorphic $h:\D\to\D$ with $h(0)=0$ and $f=F\circ h$.  The covering map
$F$ is locally univalent, and Schwarz's lemma gives $\abs{h'(0)} \leq 1$, so
$\abs{F'(0)}=\abs{f'(0)}/\abs{h'(0)}\geq1$.  Applying the definition of
$\Binf$ to $(F-F(0))/F'(0)$ gives $r(F)\geq\abs{F'(0)}\,\Binf\geq\Binf$.
Every disc on which an inverse branch of $F$ is defined lies in
$F(\D)=\Omega=f(\D)$, so $\lambda(f)\geq r(F)\geq\Binf$, and the infimum
over $f$ gives (b).
\end{proof}

Throughout the rest of the paper, $f$ satisfies~\eqref{eq:normalized-class}
and is locally univalent.  Since $f'$ has no zero and $f'(0)=1$, there is a unique
holomorphic logarithm $g=\log f'$ with $g(0)=0$.

\begin{lemma}\label{lem:logbody}
The logarithm $g$ satisfies
\[
  g'(0)=0,\qquad
  \Real g(z) \leq q(\abs{z}),
\]
where $q(r) = -\log(1-r^2)$. Consequently,
\[
  g(z)=\sum_{k\geq 2}c_kz^k,\qquad \abs{c_2} \leq 1.
\]
After rotations and conjugation we may assume
\begin{equation}\label{eq:normalization}
 c_2=a\in[0,1],\qquad c_3=p+iv,\qquad v\geq 0.
\end{equation}
\end{lemma}

\begin{proof}
  Taking logarithms in~\eqref{eq:normalized-class} gives the asserted
  inequality.  Comparing its expansion with
  $q(r)=r^2+O(r^4)$ on every ray first forces the linear coefficient of $g$ to
  vanish and then gives $\abs{c_2} \leq 1$.  The coefficient body is invariant
  under $c_k\mapsto e^{ik\alpha}c_k$ and under conjugation.  A rotation makes
  $c_2$ real and non-negative; the remaining rotation by $\pi$ changes the sign
  of $c_3$ and makes $\Imag c_3\geq 0$.
\end{proof}

All the coefficients of $g$ are bounded by the following one-parameter
family of estimates, from which every universal coefficient bound in this
paper is obtained by a choice of the parameter.

\begin{lemma}\label{lem:coeff}
For every $k\geq 2$ and every $s\in(0,1)$,
\begin{equation}\label{eq:caratheodory}
  \abs{c_k} \leq \frac{2q(s)}{s^k},
  \qquad\text{where } q(s)=-\log(1-s^2).
\end{equation}
\end{lemma}

\begin{proof}
Fix $s$ and put $P(z)=q(s)-g(sz)$ for $z\in\D$.  Since $q$ is increasing,
\eqref{eq:log-body-intro} gives
$\Real g(sz) \leq q(s\abs{z}) \leq q(s)$, so $\Real P \geq 0$ on $\D$, and
$P(0)=q(s)$ is real and positive.  Carath\'eodory's coefficient estimate
states that the Taylor coefficients $p_1,p_2,\ldots$ of a holomorphic
function with non-negative real part on $\D$ satisfy
$\abs{p_k} \leq 2\Real p_0$.  Here $p_k=-c_ks^k$, so
$\abs{c_k}s^k \leq 2q(s)$.
\end{proof}

Taking $s^2=k/(k+2)$ in~\eqref{eq:caratheodory} gives the universal bound
\begin{equation}\label{eq:universal-bound}
 \abs{c_k} \leq  B_k:=
  2\log\frac{k+2}{2}\left(\frac{k+2}{k}\right)^{k/2}
   \leq  e(k+2),\qquad k\geq 2,
\end{equation}
the last inequality because $(1+2/k)^{k/2} \leq e$ and
$2\log m \leq 2m$ for $m=(k+2)/2$.  This choice of $s$ is convenient rather
than optimal, but the loss is immaterial: in the certificate the bounds
$B_k$ are summed explicitly through $k=220$ and the elementary bound on the
right is summed geometrically thereafter.  Other values of $s$
in~\eqref{eq:caratheodory} are used where a sharper bound is needed on a
fixed finite range of $k$, as in the proof of Lemma~\ref{lem:quadratic}
below.

The following correlation between $a$ and $p$ is important near the edge
$a=1$.

\begin{lemma}[quadratic rigidity]\label{lem:quadratic}
Under~\eqref{eq:normalization},
\begin{equation}\label{eq:quadratic}
  p^2 \leq 32(1-a).
\end{equation}
Equivalently, for every real $t$,
\begin{equation}\label{eq:tangent}
  32a+2tp \leq 32+t^2.
\end{equation}
\end{lemma}

\begin{proof}
We first record the coefficient bounds that the argument uses.  Taking
$s=4/5$ in~\eqref{eq:caratheodory} gives
$\abs{c_k} \leq 2q(4/5)\,(5/4)^k$ with $2q(4/5)=2\log(25/9)=2.0433\ldots$, so
\begin{equation}\label{eq:ck-at-45}
  \abs{c_3}<4,\qquad \abs{c_k}<5\left(\tfrac54\right)^{k-4}\quad(k\geq 4),
\end{equation}
the second bound because $2q(4/5)\,(5/4)^4=4.9886\ldots<5$.

Now fix $0<x \leq 1/4$ and apply~\eqref{eq:log-body-intro} at the two real
points $x$ and $-x$.  Since $c_2=a$ and $\Real c_3=p$,
\[
  ax^2+px^3+\sum_{k\geq 4}\Real c_k\,x^k \leq  q(x),
  \qquad
  ax^2-px^3+\sum_{k\geq 4}\Real c_k(-x)^k \leq  q(x),
\]
and bounding each tail term by $-\abs{c_k}x^k$ from below turns both into
\[
  \abs{p}\,x^3 \leq  q(x)-ax^2+\sum_{k\geq 4}\abs{c_k}x^k .
\]
Here $q(x)=x^2+\sum_{j\geq 2}x^{2j}/j$, so
$q(x)-ax^2 \leq (1-a)x^2+\tfrac12x^4/(1-x^2)$, while~\eqref{eq:ck-at-45} gives
$\sum_{k\geq 4}\abs{c_k}x^k \leq 5x^4/(1-5x/4)$.  Dividing by $x^3$,
\begin{equation}\label{eq:p-vs-x}
  \abs{p} \leq \frac{1-a}{x}+
  x\left(\frac{1}{2(1-x^2)}+\frac5{1-5x/4}\right)
   \leq \frac{1-a}{x}+8x,
\end{equation}
the last step because both terms in the bracket increase with $x$ and their
sum at $x=1/4$ is $\tfrac8{15}+\tfrac{80}{11}=\tfrac{1288}{165}<8$.

It remains to optimize~\eqref{eq:p-vs-x} in $x$.  If $a=1$, letting
$x\to0^+$ gives $p=0$, which is~\eqref{eq:quadratic}.  If $1/2 \leq  a<1$, the
right side of~\eqref{eq:p-vs-x} is minimized at $x=\sqrt{(1-a)/8}$, which is
admissible because $a\geq 1/2$ forces $x \leq 1/4$; the minimum value is
$2\sqrt{8(1-a)}$, so $p^2 \leq 32(1-a)$.  If $a \leq 1/2$, then
$32(1-a)\geq 16$, while $\abs{p} \leq \abs{c_3}<4$ by~\eqref{eq:ck-at-45}, so
$p^2<16 \leq 32(1-a)$.  This proves~\eqref{eq:quadratic}.
\end{proof}

Geometrically, $\eqref{eq:quadratic}$ describes a convex region in the
$(a,p)$ plane, and~\eqref{eq:tangent} is its family of supporting lines, the
line for a given $t$ touching the parabola $p^2=32(1-a)$ at $p=t$.  Passing
from the quadratic constraint to these linear cuts therefore loses nothing,
which is what makes them admissible in the linear relaxation of
Section~\ref{sec:relax}.

In particular, on $a\geq 7/8$ one has $\abs{p} \leq 2$.  This is the only
non-box contraction used by the final cover.  The verifier rechecks the three
scalar estimates used in the proof of Lemma~\ref{lem:quadratic}; it does not
accept the contraction merely because a floating-point program reports an
empty region.

\section{Moment inequalities and LP dual certificates}
\label{sec:relax}

For fixed $0<s<1$, \eqref{eq:log-body-intro} says that
\[
  u(\theta)=q(s)-\Real g(se^{i\theta})
\]
is non-negative on the circle.  Writing
$\Real g(se^{i\theta})=\tfrac12\sum_{k\geq 2}\bigl( c_k s^k e^{ik\theta}
+\overline{c}_k s^k e^{-ik\theta} \bigr)$ and reading off Fourier coefficients
$\hat u(n)=\frac{1}{2\pi} \int_0^{2\pi} u(\theta) e^{-in\theta}\,d\theta$ gives
\begin{equation}\label{eq:fourier}
   \hat u(0)=q(s),\qquad \hat u(\pm 1)=0,\qquad
  \hat u(k)=-\frac{c_k s^k}{2},\quad \hat u(-k)=\overline{\hat u(k)}
  \qquad(k\geq 2).
\end{equation}
A non-negative function has positive semi-definite Toeplitz forms, and
testing that form against an arbitrary $w=(w_0,\ldots,w_m)\in\C^{m+1}$ gives
the following necessary inequality.

\begin{lemma}\label{lem:moment}
For every $m\geq 2$, $s\in(0,1)$, and $w\in\C^{m+1}$,
\begin{equation}\label{eq:moment-cut}
 \Real\sum_{k=2}^{m} c_k s^k
       \sum_{j=0}^{m-k} \overline{w}_{j+k} w_j
  \leq  q(s)\sum_{j=0}^{m}\abs{w_j}^2.
\end{equation}
No coefficient-tail term occurs in~\eqref{eq:moment-cut}.
\end{lemma}

\begin{proof}
Put $W(\theta)=\sum_{j=0}^{m} w_j e^{ij\theta}$.  Since $u \geq  0$,
\[
  0 \leq \frac{1}{2\pi}\int_0^{2\pi} \abs{W(\theta)}^2 u(\theta)\,d\theta
    = \sum_{j,l=0}^{m} w_j\overline{w}_l\,\hat u(l-j).
\]
By~\eqref{eq:fourier} the diagonal contributes
$q(s) \sum_j \abs{w_j}^2$, the terms with $\abs{l-j}=1$ vanish, and the terms
with $l-j=k \geq  2$ and with $j-l=k \geq  2$ are complex conjugates of each other,
so together they contribute
$-\Real\sum_{k=2}^{m} c_k s^k \sum_{j=0}^{m-k}\overline{w}_{j+k} w_j$.
Rearranging gives~\eqref{eq:moment-cut}.  Only $\hat u(n)$ with
$\abs{n} \leq  m$ occur, because $\abs{W}^2$ has spectrum in $[-m,m]$; by
\eqref{eq:fourier} these involve $c_2, \ldots, c_m$ and nothing else, which is
the last assertion.
\end{proof}

Numerical eigensolvers are convenient for finding vectors $w$ that cut off
LP minimizers, but their spectral claims are irrelevant to the proof. The
checker reconstructs~\eqref{eq:moment-cut} directly for each stored vector.
The relaxation also includes the tangent cuts~\eqref{eq:tangent} and the
pointwise consequences of~\eqref{eq:log-body-intro}
\begin{equation}\label{eq:point-cut}
 \Real\sum_{k=2}^{m} c_k \, \frac{z^k}{\abs{z}^2}
  \leq \frac{-\log(1-\abs{z}^2) + T_m(\abs{z})}{\abs{z}^2},
 \qquad 0<\abs{z}<1,
\end{equation}
where
\begin{equation}\label{eq:tail}
  T_m(r)=\sum_{k=m+1}^{220}B_kr^k
  +e\,r^{221}\Bigl(\frac{223}{1-r}+\frac{r}{(1-r)^2}\Bigr)
  \geq\sum_{k>m}\abs{c_k}r^k
\end{equation}
is the tail of the universal bounds~\eqref{eq:universal-bound}: the bounds
$B_k$ are summed explicitly through $k=220$, and the closed form of
$\sum_{k\geq221}e(k+2)r^k$ accounts for the rest.

Write the retained real coordinates as
\[
  x=(c_2, \Real c_3, \ldots, \Real c_m, \Imag c_3, \ldots, \Imag c_m),
\]
and the necessary cuts as $Ax \leq  b$.  If the current coefficient box is
$l \leq  x \leq  h$, the following elementary dual estimate is the basis of the
certificate.

\begin{lemma}[residual dual bound]\label{lem:dual}
For an objective $d \in \R^{2m-3}$ and any $y \geq  0$,
\begin{equation}\label{eq:dual}
  d^Tx \geq  -y^Tb + \sum_j\min \{(d+A^Ty)_j l_j, (d+A^Ty)_j h_j \}.
\end{equation}
\end{lemma}

\begin{proof}
Since $Ax \leq  b$ and $y\geq 0$, it follows that $d^Tx \geq  (d+A^Ty)^Tx - y^Tb$.
Minimize each remaining coordinate over its interval.
\end{proof}

Thus neither exact stationarity nor exact optimality is needed. In the verifier,
every binary64 multiplier is interpreted as an exact rational number, and the
right side of~\eqref{eq:dual} is evaluated with outward-rounded interval
arithmetic.

\section{From logarithmic lower bounds to schlicht discs}
\label{sec:paths}

We use the Liu--Minda distortion estimate~\cite[Theorem~1(a)]{LiuMinda}
in the form
\begin{equation}\label{eq:LM}
 \abs{f'(z)}\geq \ell(\abs{z}),\qquad\text{where }
 \ell(r)=\frac{\exp(-2r/(1-r))}{(1-r)^2},
\end{equation}
and write
\begin{equation}\label{eq:E}
 E(r)=\int_0^r\ell(t)\,dt
 =\frac{1-\exp(-2r/(1-r))}{2}.
\end{equation}
In particular, $E(1-)=1/2$.

\subsection{The escaping path}\label{sub:escape}

The mechanism that converts lower bounds for $\abs{f'}$ into a schlicht disc
is that of Chen--Shiba~\cite{ChenShiba} in the form it is used by the
certificate.

Since $f$ is locally univalent with $f(0)=0$, an inverse branch $\varphi$ of
$f$ is defined near the origin with $\varphi(0)=0$.  Fix an image direction
$\theta$ and continue $\varphi$ analytically along the ray
$t\mapsto te^{i\theta}$, $t\geq 0$.  Let $R_\theta\in(0,\infty]$ be the
supremum of those $T$ for which the continuation exists on $[0,T]$, and for
$0 \leq  t<R_\theta$ write
\[
  \gamma_\theta(t)=\varphi(te^{i\theta})
\]
for the resulting path in $\D$, the \emph{lift} of the ray.  By construction
the branch is single-valued on the disc $\abs{w}<\inf_\theta R_\theta$, so
\begin{equation}\label{eq:rf-lower}
  r(f)\geq \inf_\theta R_\theta .
\end{equation}

\begin{lemma}[escaping path]\label{lem:escape}
Let $\omega:[0,1)\to[0,\infty)$ be continuous except at finitely many
points, and suppose that
$\abs{f'(\gamma_\theta(t))}\geq \omega(\abs{\gamma_\theta(t)})$ for
$0 \leq  t<R_\theta$.  Then
\[
  R_\theta\geq \int_0^1\omega(s)\,ds .
\]
\end{lemma}

\begin{proof}
We may assume $R_\theta<\infty$.  The lift then leaves every compact subset
of $\D$: otherwise $\gamma_\theta(t_n)\to z_*\in\D$ for some
$t_n\nearrow R_\theta$ and some subsequence, so that
$f(z_*)=R_\theta e^{i\theta}$.  Since $f'(z_*)\neq0$, $f$ is univalent on a
disc $U$ about $z_*$ and $f(U)$ contains a disc about $R_\theta e^{i\theta}$.
For large $n$ the inverse branch of $f|_U$ and the continuation of $\varphi$
both send $t_ne^{i\theta}$ to $\gamma_\theta(t_n)$, hence agree near that
point; so the former continues the latter past $R_\theta$, contradicting the
definition of $R_\theta$.

Consequently $\varrho(t)=\abs{\gamma_\theta(t)}$ satisfies $\varrho(0)=0$ and
$\varrho(t)\to1$ as $t\nearrow R_\theta$; moreover $\varrho$ is Lipschitz on
every $[0,T]$ with $T<R_\theta$, because $\gamma_\theta$ is smooth there.
Suppose first that $\omega$ is continuous and put $\Psi(s)=\int_0^s\omega$.
Since $f\circ\gamma_\theta$ is the segment from $0$ to $R_\theta e^{i\theta}$,
for $T<R_\theta$
\[
  T=\int_{\gamma_\theta|_{[0,T]}}\abs{f'(z)}\,\abs{dz}
   \geq \int_0^T\omega(\varrho(t))\,\abs{\varrho'(t)}\,dt
   \geq \Psi(\varrho(T)),
\]
the first inequality by the hypothesis along the lift together with
$\abs{d\abs{z}} \leq \abs{dz}$, the second because
$\int_0^T\omega(\varrho)\abs{\varrho'}\geq \bigl|\int_0^T\omega(\varrho)\varrho'\bigr|
=\Psi(\varrho(T))$, where the last equality holds since $\Psi\circ\varrho$
is Lipschitz with derivative $\omega(\varrho)\varrho'$ almost everywhere.
Letting $T\nearrow R_\theta$ gives $R_\theta\geq \Psi(1^-)$.

In general put
$\omega_n(s)=\inf_{u\in[0,1)}\bigl(\omega(u)+n\abs{s-u}\bigr)$.  Then
$\omega_n$ is continuous, $0 \leq \omega_n \leq \omega$, and $\omega_n(s)$
increases to $\omega(s)$ at every point of continuity of $\omega$, hence
almost everywhere.  The continuous case applied to $\omega_n$ and monotone
convergence give the claim.
\end{proof}

Taking $\omega=\ell$ gives the baseline $R_\theta\geq  E(1-)=1/2$
for every $\theta$, which with~\eqref{eq:rf-lower} is the classical bound
$\Binf\geq 1/2$.  Every improvement in this paper comes from replacing $\ell$
by a larger $\omega$ on part of the radial range.  The link to
Section~\ref{sec:relax} is the identity
\begin{equation}\label{eq:modulus-log}
  \abs{f'(z)}=\exp\Real g(z),
\end{equation}
which turns a certified lower bound for the linear functional $\Real g(z)$
into a lower bound for $\abs{f'}$.  The hypothesis of Lemma~\ref{lem:escape}
is needed only along the lift, and Section~\ref{sub:angles} shows that the
part of an annulus that a lift with a given image direction can meet is a
sector of source angles.

\begin{corollary}[excess over the baseline]\label{cor:excess}
Let $0 \leq  r_0<r_1<\cdots<r_n<1$, and for each $k$ let $\omega_k$ be
continuous on $[r_k,r_{k+1}]$ with $\abs{f'(z)}\geq \omega_k(\abs{z})$ for
every point $z$ of the lift $\gamma_\theta$ with
$r_k \leq \abs{z} \leq  r_{k+1}$.  Then
\[
  R_\theta\geq \frac{1}{2}+\sum_{k=0}^{n-1}d_k,
  \qquad
  d_k=\max\Bigl(0,\int_{r_k}^{r_{k+1}}\omega_k(s)\,ds
       -\bigl(E(r_{k+1})-E(r_k)\bigr)\Bigr).
\]
\end{corollary}

\begin{proof}
Apply Lemma~\ref{lem:escape} with $\omega=\max(\ell,\omega_k)$ on
$[r_k,r_{k+1})$ and $\omega=\ell$ on $[0,r_0)$ and on $[r_n,1)$; this
$\omega$ is continuous except possibly at the grid points.  Then use
$\int\max(\ell,\omega_k)\geq \max\bigl(\int\ell,\int\omega_k\bigr)$ on each
$[r_k,r_{k+1}]$ together with~\eqref{eq:E}.
\end{proof}

%% Schematic of the escaping-path mechanism (Section 4.1).
\begin{figure}[ht]
\centering
\begin{tikzpicture}[x=1cm,y=1cm,>=latex]
  % ---- source disc ----
  \begin{scope}
    \fill[black!15] (0,0) -- (12:2.2) arc[start angle=12,end angle=78,radius=2.2] -- cycle;
    \foreach \r in {0.75,1.25,1.7} \draw[black!58,dashed,line width=.4pt] (0,0) circle (\r);
    \draw[black!70,line width=.5pt] (0,0) circle (2.2);
    \draw[covhue,line width=.9pt,smooth,tension=.7]
      plot coordinates {(0,0) (0.421,0.354) (0.602,0.860) (1.188,0.832)
                        (0.980,1.569) (1.640,1.640)};
    \fill (0,0) circle (1.3pt);
    \node[font=\scriptsize,below left,inner sep=1.5pt] at (0,0) {$0$};
    \node[font=\scriptsize,covtext,right,inner sep=2pt] at (1.24,0.70) {$\gamma_\theta$};
    \node[font=\scriptsize,black!65,left,inner sep=1.5pt] at (-0.02,1.25) {$r_k$};
    \node[font=\scriptsize] at (0.95,2.33) {$\Omega$};
    \node[font=\scriptsize,below] at (0,-2.62) {source disc $\D$};
  \end{scope}
  % ---- image plane ----
  \begin{scope}[xshift=6.9cm]
    \fill[black!15] (0,0) -- (40:2.35) arc[start angle=40,end angle=60,radius=2.35] -- cycle;
    \draw[black!68,dashed,line width=.45pt] (0,0) circle (1.3);
    \draw[covhue,line width=.9pt,->] (0,0) -- (50:1.98);
    \draw[covhue,line width=.9pt] (50:1.90) -- (50:2.06);
    \node[font=\scriptsize,covtext,below right,inner sep=2pt] at (50:0.95) {$te^{i\theta}$};
    \fill (0,0) circle (1.3pt);
    \node[font=\scriptsize,below left,inner sep=1.5pt] at (0,0) {$0$};
    \fill[black!75] (0,0.45) circle (1.3pt);
    \node[font=\scriptsize,above left,inner sep=1pt] at (0,0.45) {$w_0=i\varepsilon$};
    \node[font=\scriptsize,black!70,below,inner sep=2pt] at (265:1.3) {$\tfrac12$};
    \node[font=\scriptsize] at (1.30,2.10) {$I_j$};
    \node[font=\scriptsize,covtext,right,inner sep=2pt] at (50:2.10) {$R_\theta$};
    \node[font=\scriptsize,below] at (0,-2.62) {image plane};
  \end{scope}
\end{tikzpicture}
\caption{The escaping-path mechanism.  An inverse branch is continued along
the image ray of direction $\theta$ until its first singularity, at distance
$R_\theta$; its lift $\gamma_\theta$ leaves every compact subset of $\D$
(Lemma~\ref{lem:escape}).  Integrating $\abs{f'}\geq\ell(\abs z)$ along the
lift gives $R_\theta\geq\frac12$, the dashed circle on the right, and any
better lower bound for $\abs{f'}$ on an annulus $r_k\leq\abs z\leq r_{k+1}$
adds its excess (Corollary~\ref{cor:excess}).  Within the annulus under consideration,
only the shaded sector $\Omega$ of source angles can be met by a lift whose image
direction lies in the bin $I_j$ (Lemma~\ref{lem:angle}).}
\label{fig:escape}
\end{figure}
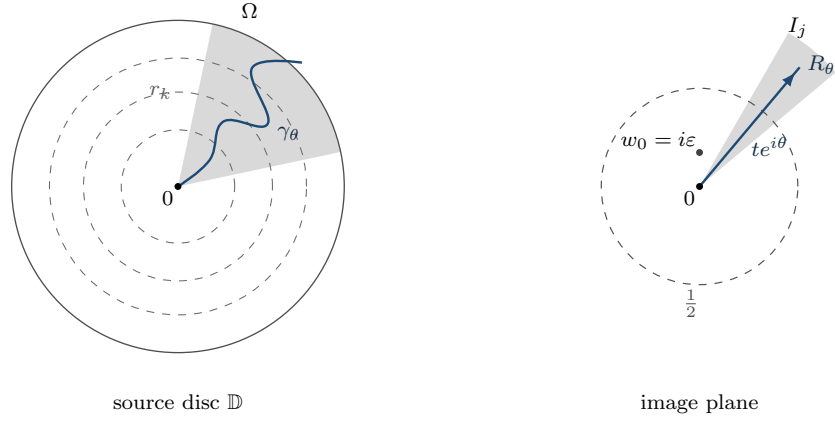

The rest of this section produces the functions $\omega_k$: Sections
\ref{sub:arcs} and~\ref{sub:curv} bound $\Real g$ from below on the circles
of the grid, Section~\ref{sub:radial} interpolates between consecutive
circles, and Section~\ref{sub:angles} identifies the sector of source angles that a
lift with a given image direction can meet.  The numbers $d_k$ of
Corollary~\ref{cor:excess} are the excesses $d_{\sigma,k,j}$ of
Definition~\ref{def:source}.

\subsection{Angular arcs}\label{sub:arcs}

At a fixed radius $r$, apply Lemma~\ref{lem:dual} to
\[
  d^T x = \Real\sum_{k=2}^{m} c_k\, r^k e^{ik\theta}.
\]
Let $C_k$ be a rigorous upper bound for $\abs{c_k}$ on the current box.
Then
\begin{equation}\label{eq:termwise}
 \left|\frac{\partial^2}{\partial\theta^2}
 \Real\sum_{k=2}^{m} c_k \, r^k e^{ik\theta}\right|
  \leq \sum_{k=2}^{m} k^2 C_k r^k .
\end{equation}
We call an estimate obtained in this way, by replacing each $\abs{c_k}$ by $C_k$
term by term, a \emph{termwise} bound. If dual witnesses give endpoint bounds
$\Lambda_\alpha,\Lambda_\beta$ on an angular arc of width $\Delta$, and
$M_\theta(r)$ is any non-negative upper bound for the second derivative on the
arc, the whole arc admits a lower bound
\begin{equation}\label{eq:arc}
  \min(\Lambda_\alpha,\Lambda_\beta)-\frac{1}{8}M_\theta(r)\Delta^2.
\end{equation}
The stored arcs have endpoints that are exact rational multiples of $2\pi$.
The checker verifies that they are contiguous and cover the required source
angle domain exactly.

\subsection{Certified curvature}\label{sub:curv}

The termwise bound~\eqref{eq:termwise} discards every cancellation among the
$c_k$, and takes the universal values $B_k$ for $C_k$ when $k\geq 4$; on a
small box it exceeds the actual second derivative several times over.  Since
it enters the allowance in~\eqref{eq:arc} directly, it is the largest single
loss in the estimate, costing an order of magnitude more on the weakest
leaves than either the infinite tail or the angular distortion
(Section~\ref{sec:search}).  The remedy is to bound the second derivative by
the same dual mechanism, in the one-sided form of the Taylor allowance.

\begin{lemma}[one-sided Taylor allowance]\label{lem:curvature}
Let $\varphi$ be twice differentiable on $[\alpha,\beta]$ with
$\varphi'' \leq  M$ there, where $M\geq0$, let $\Delta=\beta-\alpha$, and let
\[
  L(\theta)=\frac{\beta-\theta}{\Delta}\varphi(\alpha)
           +\frac{\theta-\alpha}{\Delta}\varphi(\beta)
\]
be the chord of $\varphi$.  Then, for $\alpha \leq \theta \leq \beta$,
\[
  \varphi(\theta)\geq  L(\theta)-\frac M2(\theta-\alpha)(\beta-\theta)
   \geq  L(\theta)-\frac18M\Delta^2
   \geq \min\bigl(\varphi(\alpha),\varphi(\beta)\bigr)-\frac18M\Delta^2 .
\]
\end{lemma}

\begin{proof}
Put $\psi(\theta)=\varphi(\theta)-L(\theta)
+\tfrac M2(\theta-\alpha)(\beta-\theta)$.  Then $\psi''=\varphi''-M \leq 0$
and $\psi(\alpha)=\psi(\beta)=0$, so $\psi$ is concave with vanishing
endpoint values and hence non-negative on $[\alpha,\beta]$; this is the
first inequality.  The second follows from
$(\theta-\alpha)(\beta-\theta) \leq \Delta^2/4$ and $M\geq0$, and the third
because $L$ is affine.
\end{proof}

The middle form is what the radial estimate~\eqref{eq:radial} uses, and the last
form is what the angular estimate~\eqref{eq:arc} uses.  The first inequality
holds for every real~$M$, but the two uniform forms need $M\geq0$. A certified
upper bound for a second derivative may well be negative, and the checker
replaces it by~$0$ before it enters an allowance.

Thus only an \emph{upper} bound for the second derivative is needed.  With
$h_m(z)=\sum_{k=2}^{m} c_k z^k$ and $z = re^{i\theta}$,
\[
  \frac{\partial^2}{\partial\theta^2}\Real h_m(z)
    = -\Real\sum_{k=2}^{m} k^2 c_k z^k,
  \qquad
  \frac{\partial^2}{\partial t^2}\Real h_m(te^{i\theta})\Big|_{t=r}
    = \Real\sum_{k=2}^{m} k(k-1) c_k \frac{z^k}{\abs{z}^2},
\]
and both right sides are linear in the retained coordinates $x$.  At every
stored point $z$ the certificate may contain a non-negative dual vector for
the objective $\Real \sum k^2 c_k z^k$ and one for the objective
$-\Real\sum k(k-1) c_k z^k/\abs{z}^2$. Lemma~\ref{lem:dual} then gives lower
bounds $\ell_\theta(z)$ and $\ell_r(z)$ for these objectives on the box, and
$U_\theta(z)=-\ell_\theta(z)$, $U_r(z)=-\ell_r(z)$ are certified upper
bounds for the two second derivatives at $z$.  Between two stored points at
the same radius, the fourth angular derivative is bounded termwise, so on an
arc of width $\Delta$
\begin{equation}\label{eq:certified-curvature}
  \sup_{\text{arc}} \frac{\partial^2}{\partial\theta^2} \Real h_m
     \leq  \max\bigl(U_\theta(\alpha),U_\theta(\beta)\bigr)
    +\frac{1}{8}\Bigl(\sum_{k=2}^{m} k^4 C_k r^k\Bigr) \Delta^2,
\end{equation}
by Lemma~\ref{lem:curvature} applied to $-\partial_\theta^2\Real h_m$.
The radial second derivative is propagated in two steps.  On an arc of
width $\Delta$ at radius $r$, the termwise bound
$\abs{\partial_\theta^2\partial_r^2\Real h_m}
 \leq \sum_{k=2}^{m}k^3(k-1)C_kr^{k-2}$ gives
\begin{equation}\label{eq:radial-arc}
  \sup_{\text{arc}}\frac{\partial^2}{\partial r^2}\Real h_m
   \leq \max\bigl(U_r(\alpha),U_r(\beta)\bigr)
   +\frac18\Bigl(\sum_{k=2}^{m}k^3(k-1)C_kr^{k-2}\Bigr)\Delta^2,
\end{equation}
and the supremum over a window of source angles at radius $r$ is the largest
of these values over the arcs that meet the window.  On an annulus $[u,v]$
the fourth radial derivative is bounded termwise, and the supremum over the
window and the annulus is at most the larger of the two circle suprema plus
$\frac{1}{8}\bigl(\sum_{k=4}^{m}k(k-1)(k-2)(k-3)C_kv^{k-4}\bigr)(v-u)^2$.
For every arc and every annulus, the checker takes as $M$ the smaller of
the termwise and the certified bound, the latter replaced by $0$ if it is
negative; both are valid.  Curvature duals are
therefore optional: a source that carries none is verified with the termwise
bound alone.  Where they are present, the two extra multipliers per stored
point are produced by two further LPs during the search and are checked like
every other dual vector.  Because the allowance is now small, it is also
worth refining the arcs in the source directions that bind after the shift
of Section~\ref{sec:shift}, and the search does so until their allowance
falls below a prescribed tolerance.

\subsection{Radial intervals}\label{sub:radial}

Both constructions below produce a lower bound $\varphi$ for $\Real g$ on an
annulus $u \leq \abs{z} \leq  v$ and then need a lower bound for
$\int_u^v e^{\varphi(t)}\,dt$, the contribution of that annulus to the length
estimate of Corollary~\ref{cor:excess}. Jensen's inequality gives
\begin{equation}\label{lem:jensen}
  \int_u^v e^{\varphi(t)}\,dt\geq
  (v-u)\exp\left(\frac1{v-u}\int_u^v\varphi(t)\,dt\right).
\end{equation}

If the source window is the full circle, subtracting the tail
$T_m$ of~\eqref{eq:tail} from the arc bounds gives lower bounds $G(u)$ and
$G(v)$ for $\Real g$ on the boundary circles of an annulus.  The harmonic
minimum principle gives
\begin{equation}\label{eq:harmonic}
  \Real g(z) \geq  G(u) + \frac{\log(\abs{z}/u)}{\log(v/u)}(G(v)-G(u)),
  \qquad\text{for $u  \leq  \abs{z}  \leq  v$}.
\end{equation}
The exponential of the right side can be integrated in closed form.  If
$P=\log(v/u)$ and $Q=G(v)-G(u)+P$, the integral equals
\begin{equation}\label{eq:annulus-integral}
 u e^{G(u)}P\frac{e^Q-1}{Q},
\end{equation}
with its removable limit at $Q=0$.  When interval arithmetic cannot decide
the sign of a ball containing $Q=0$, this closed form is unusable and the
verifier falls back on~\eqref{lem:jensen}, applied to the right side
of~\eqref{eq:harmonic}.

On a window of source angles that is not the full circle, the harmonic
minimum principle would need bounds on the two radial sides of the annular
sector as well, and those are not available; the bounds on the circular
arcs are.  Instead, a radial second-derivative bound gives, for every
$\theta$ in the window,
\begin{equation}\label{eq:radial}
  \Real h_m(te^{i\theta}) \geq
  \frac{v-t}{v-u}\,\Lambda_u + \frac{t-u}{v-u}\,\Lambda_v
  -\frac{1}{8}M_r(v)(v-u)^2,
\end{equation}
where $h_m=\sum_{k=2}^{m} c_k z^k$, where $\Lambda_u$ and $\Lambda_v$ are
lower bounds for $\Real h_m$ on the window at the radii $u$ and $v$, namely
the minima of the arc bounds~\eqref{eq:arc} over the arcs that meet the
window, and where $M_r(v)\geq0$ is an upper bound for the radial second
derivative on the window and the annulus, either the termwise bound
$\sum_{k=2}^{m} k(k-1) C_k v^{k-2}$ or the certified bound of
Section~\ref{sub:curv}.  Here the right side of~\eqref{eq:radial} is affine
in $t$, so~\eqref{lem:jensen} applies with the value of that affine
function at the midpoint, after the tail $T_m(v)$ has been subtracted.  This
estimate is used for every source and every image bin, on the window of
source angles that the bin can meet; a full-circle source contributes the
harmonic estimate in addition, in every image direction, and the profile of
Definition~\ref{def:source} takes the better of the two.  In either case,
only the positive excess over $E(v)-E(u)$ enters, as in
Corollary~\ref{cor:excess}.

\subsection{Source and image directions}\label{sub:angles}

The LP bounds concern source angles, whereas the lifted inverse path is
naturally indexed by its image direction.  The two are related by a bound on
$\arg(f(z)/z)$, which we now derive.  Its only inputs are the upper bounds
$C_2\geq \abs{c_2}$ and $C_3\geq \abs{c_3}$ valid on the current box, so the
sharper those bounds are, the narrower the sector a lift can occupy.

\begin{lemma}[angular distortion]\label{lem:angle}
Let $0<r \leq  s<1$ and put
\[
  K=\frac{\artanh s}{s},\qquad \gamma=K^{-1},\qquad
  b=\frac{C_2s^2}{3K(1-\gamma^2)},\qquad \tau=\frac rs .
\]
Assume $b<1$ and set
\[
 \kappa=\min\left(1,\frac{C_3s^3}{4K(1-\gamma^2)(1-b^2)}\right),\qquad
 \eta=\frac{\kappa+\tau}{1+\kappa\tau},\qquad
 \delta=\tau^2\,\frac{b+\tau\eta}{1+b\tau\eta}.
\]
If $\delta(K-K^{-1}) \leq 1-\delta^2$, then
\begin{equation}\label{eq:angle-distortion}
 \left|\arg\frac{f(z)}z\right| \leq \beta(r):=
 \arcsin\frac{\delta(K-K^{-1})}{1-\delta^2},
 \qquad \abs{z} \leq  r.
\end{equation}
\end{lemma}

\begin{proof}
Write $H(\zeta)=f(s\zeta)/(s\zeta K)$, so that $\arg(f(z)/z)=\arg H(\zeta)$
for $z=s\zeta$, because $K>0$.

\emph{$H$ is a self-map of $\D$.}  Integrating~\eqref{eq:normalized-class}
along a radius gives $\abs{f(w)} \leq \artanh \abs{w}$ for $w\in\D$.
Since $\artanh $ is convex on $[0,1)$ and vanishes at $0$, the
quotient $x\mapsto\artanh (x)/x$ increases, so
$\artanh (s\abs{\zeta}) \leq \abs{\zeta}\artanh (s)$ for
$\abs{\zeta} \leq 1$.  Hence
$\abs{f(s\zeta)} \leq \abs{\zeta}\artanh (s)=s\abs{\zeta} K$, that is
$\abs{H} \leq 1$.

\emph{Expansion.}  From $f'=\exp g$ and $g=c_2z^2+c_3z^3+\cdots$ we get
$f'(z)=1+c_2z^2+c_3z^3+O(z^4)$, hence
$f(z)/z=1+\tfrac13c_2z^2+\tfrac14c_3z^3+O(z^4)$ and
\[
  H(\zeta)=\gamma+\frac{c_2s^2}{3K}\zeta^2+\frac{c_3s^3}{4K}\zeta^3+O(\zeta^4).
\]
In particular $H(0)=\gamma$ and $H'(0)=0$.

\emph{Normalization.}  Let $\mu(w)=(w-\gamma)/(1-\gamma w)$, an automorphism
of $\D$ with $\mu(\gamma)=0$ and $\mu'(\gamma)=(1-\gamma^2)^{-1}$, and put
$\chi=\mu\circ H$ and $\psi(\zeta)=\chi(\zeta)/\zeta^2$.  Then $\chi$ maps
$\D$ into $\D$ and vanishes to order at least $2$ at $0$, so $\psi$ is
holomorphic and, by the Schwarz lemma applied twice, maps $\D$ into $\D$.
Comparing coefficients,
\[
  \psi(0)=\frac{c_2s^2}{3K(1-\gamma^2)},\qquad
  \psi'(0)=\frac{c_3s^3}{4K(1-\gamma^2)},
\]
so $\abs{\psi(0)} \leq  b$ and $\abs{\psi'(0)} \leq  C_3s^3/(4K(1-\gamma^2))$.

\emph{Schwarz--Pick.}  Let $\Phi=\mu_{\psi(0)}\circ\psi$ with
$\mu_a(w)=(w-a)/(1-\bar aw)$.  Then $\Phi:\D\to\D$, $\Phi(0)=0$ and
$\abs{\Phi'(0)}=\abs{\psi'(0)}/(1-\abs{\psi(0)}^2) \leq \kappa$, the cap at $1$
being harmless since $\abs{\Phi'(0)} \leq 1$ anyway.  Applying Schwarz--Pick to
the self-map $\Phi(\zeta)/\zeta$ of $\D$ gives
\[
  \left|\frac{\Phi(\zeta)}{\zeta}\right|
    \leq \frac{\abs{\Phi'(0)}+\abs{\zeta}}{1+\abs{\Phi'(0)}\abs{\zeta}}
    \leq \frac{\kappa+\tau}{1+\kappa\tau}=\eta,
  \qquad \abs{\zeta} \leq \tau,
\]
the last step because the middle quantity increases in both
$\abs{\Phi'(0)}$ and $\abs{\zeta}$.  Inverting $\mu_{\psi(0)}$ and using that
$x\mapsto(x+b)/(1+bx)$ is increasing,
\[
  \abs{\psi(\zeta)}
     \leq  \frac{\tau\eta+\abs{\psi(0)}}{1+\abs{\psi(0)}\tau\eta}
     \leq  \frac{\tau\eta+b}{1+b\tau\eta},
  \qquad\text{hence}\qquad
    \abs{\chi(\zeta)} = \abs{\zeta}^2\abs{\psi(\zeta)} \leq \delta
\]
for $\abs{\zeta} \leq \tau$.

\emph{From $\chi$ to the argument.}  Finally $H=\mu^{-1}\circ\chi$, and
$\mu^{-1}(w)=(w+\gamma)/(1+\gamma w)$ maps the disc $\abs{w} \leq \delta$ onto
the disc with diameter the segment from $(\gamma-\delta)/(1-\gamma\delta)$ to
$(\gamma+\delta)/(1+\gamma\delta)$, that is onto the disc of center
$c=\gamma(1-\delta^2)/(1-\gamma^2\delta^2)$ and radius
$\varrho=\delta(1-\gamma^2)/(1-\gamma^2\delta^2)$.  The largest argument
attained on a disc of center $c>0$ and radius $\varrho \leq  c$ is
$\arcsin(\varrho/c)$, and
\[
  \frac{\varrho}{c}=\frac{\delta(1-\gamma^2)}{\gamma(1-\delta^2)}
   =\frac{\delta(K-K^{-1})}{1-\delta^2},
\]
since $(1-\gamma^2)/\gamma=K-K^{-1}$.  The assumption
$\delta(K-K^{-1}) \leq 1-\delta^2$ is exactly $\varrho \leq  c$, so
$\abs{\arg H(\zeta)} \leq \beta(r)$ for $\abs{\zeta} \leq \tau$, i.e.\ for
$\abs{z} \leq  r$.
\end{proof}

The certificate takes the best valid estimate from a finite list of exact
rational values of $s$; if none applies it safely uses $\beta=\pi$.

Divide the image-angle circle into bins $I_j$.  By Lemma~\ref{lem:angle},
the part of a lift with image angle $\theta\in I_j$ that lies in an annulus
with outer radius $v$ has source angles in $I_j+[-\beta(v),\beta(v)]$.  For
a sectorial source, an annulus contributes only if this window is contained
in the source window; otherwise its excess is taken to be zero.
Equations~\eqref{eq:arc}--\eqref{eq:radial} on these windows, together with
Corollary~\ref{cor:excess}, give a lower bound $\rho_j$ for $R_\theta$ for
every $\theta\in I_j$.

A group of sources sharing one radial grid may moreover be combined annulus
by annulus, and not merely compared.

\begin{lemma}[annulus-wise maximum]\label{lem:mixing}
Let $\mathcal S$ be a finite set of sources sharing the radial grid
$0 \leq  r_0<\cdots<r_n<1$, and suppose that each $\sigma\in\mathcal S$
supplies on each annulus a continuous $\omega_{\sigma,k}$ with
$\abs{f'(z)}\geq \omega_{\sigma,k}(\abs{z})$ for every point $z$ of the lift
$\gamma_\theta$ with $r_k \leq \abs{z} \leq  r_{k+1}$, with excess
$d_{\sigma,k}$ as in Corollary~\ref{cor:excess}.  Then
\begin{equation}\label{eq:annulus-mixing}
 R_\theta\geq \frac{1}{2}+\sum_{k=0}^{n-1}\max_{\sigma\in\mathcal S}d_{\sigma,k}.
\end{equation}
\end{lemma}

\begin{proof}
On $[r_k,r_{k+1}]$ put $\omega_k=\max_{\sigma\in\mathcal S}\omega_{\sigma,k}$,
again continuous and again a lower bound for $\abs{f'}$ there.  Write
$\Delta_kE=E(r_{k+1})-E(r_k)$.  Corollary~\ref{cor:excess}, applied to these
$\omega_k$, gives $R_\theta\geq \frac{1}{2}+\sum_kd_k$ with
\[
 d_k=\max\Bigl(0,\int_{r_k}^{r_{k+1}}\!\!\max_\sigma\omega_{\sigma,k}
      -\Delta_kE\Bigr)
  \geq \max\Bigl(0,\max_\sigma\int_{r_k}^{r_{k+1}}\!\!\omega_{\sigma,k}
      -\Delta_kE\Bigr)
  =\max_\sigma d_{\sigma,k},
\]
the inequality by monotonicity of the integral and the last equality
because $x\mapsto\max(0,x-\Delta_kE)$ is non-decreasing.
\end{proof}

The hypothesis that the grids agree is what allows the maximum to be taken
inside the sum: the excesses may then be compared annulus by annulus against
one and the same partition of the radial range.  Accordingly, the certificate
builds~\eqref{eq:annulus-mixing} only from sources with the identical stored
binary64 grid; estimates from different grids remain alternative full
profiles and are never added, and the original source profiles and their
origin-radius bounds are retained as fallbacks.  The combination is useful
because different dual witnesses can be strong on different radial portions
of the same path.

\subsection{The shifted disc}\label{sec:shift}

The bounds $R_\theta\geq\rho_j$ concern a schlicht disc centred at $0$; a
disc with a different centre can be larger.  Let
\[
  S=\{te^{i\theta}: t\geq0,\ \theta\in\R,\ t<R_\theta\}
\]
be the union of the segments along which the inverse germ at $0$ has been
continued.

\begin{lemma}[radial continuation]\label{lem:star}
The set $S$ is open and star-shaped with respect to $0$, and
$\Phi(te^{i\theta})=\gamma_\theta(t)$ defines a holomorphic function
$\Phi:S\to\D$ with $\Phi(0)=0$ and $f\circ\Phi=\mathrm{id}$.  Consequently,
if $S$ contains the open disc of radius $\rho$ about $w_0$, then
$r(\Phi(w_0),f)\geq\rho$.
\end{lemma}

\begin{proof}
Star-shapedness is clear from the definition.  Let $w_0\in S$.  The
continuation of $\varphi$ along the segment $[0,w_0]$ is carried by finitely
many open discs $D_1,\ldots,D_N$ and inverse branches $\varphi_\nu$ of $f$
on $D_\nu$, with $\varphi_1=\varphi$ near $0$,
$\varphi_\nu=\varphi_{\nu+1}$ on $D_\nu\cap D_{\nu+1}$, and points
$0=a_0<a_1<\cdots<a_N=1$ such that
$\{sw_0:a_{\nu-1} \leq  s \leq  a_\nu\}\subset D_\nu$.  By compactness there
is $\eta>0$ such that $\{sw:a_{\nu-1} \leq  s \leq  a_\nu\}\subset D_\nu$ for
every $\nu$ whenever $\abs{w-w_0}<\eta$.  For such $w$ the same discs and
branches continue $\varphi$ along $[0,w]$, and since continuation along a
path is unique, this is the radial continuation in the direction of $w$; as
$D_N$ is open, it extends slightly beyond $w$.  Hence $R_{\arg w}>\abs{w}$,
so $S$ is open, and $\Phi=\varphi_N$ near $w_0$, so $\Phi$ is holomorphic;
$f\circ\Phi=\mathrm{id}$ and $\Phi(0)=0$ hold by construction.  If $S$
contains the open disc of radius $\rho$ about $w_0$, the restriction of
$\Phi$ to that disc is an inverse branch of $f$ centred at
$f(\Phi(w_0))=w_0$, which is the last claim.
\end{proof}

Now choose a new image centre $w_0=i\varepsilon$ with
$0 \leq \varepsilon < \frac{1}{2}$.  Let $\rho_j\geq\frac{1}{2}$ be lower bounds for
$R_\theta$ on the bins $I_j$, and let $S_j \leq 1$ be an upper bound for
$\sin\theta$ on $I_j$.  A point $w=te^{i\theta}\notin S$ with $\theta\in I_j$
has $t\geq R_\theta\geq\rho_j$, and
\[
  \abs{w-i\varepsilon}^2=t^2+\varepsilon^2-2t\varepsilon\sin\theta
  \geq t^2+\varepsilon^2-2t\varepsilon S_j
  \geq \rho_j^2+\varepsilon^2-2\rho_j\varepsilon S_j ,
\]
the last step because $t\mapsto t^2-2t\varepsilon S_j$ is non-decreasing for
$t\geq\varepsilon S_j$ and
$t\geq\rho_j\geq\frac{1}{2}\geq\varepsilon\geq\varepsilon S_j$.  Hence every
point outside $S$ with argument in $I_j$ is at distance at least
\begin{equation}\label{eq:shifted}
 \sqrt{\rho_j^2+\varepsilon^2-2\rho_j\varepsilon S_j}
\end{equation}
from $i\varepsilon$.  The open disc about $i\varepsilon$ whose radius is the
minimum of~\eqref{eq:shifted} over all bins is therefore contained in $S$,
and by Lemma~\ref{lem:star} that minimum is the radius of a schlicht disc.
The floating-point search proposes~$\varepsilon$; the checker treats its
stored decimal expansion as an exact rational, accepts it only if
$0 \leq \varepsilon<\frac14$, and recomputes~\eqref{eq:shifted}.

\subsection{Sources, profiles, and the bound on a box}

The constructions of Sections~\ref{sub:arcs}--\ref{sec:shift} enter the rest
of the paper only through one bundled object.

\begin{definition}\label{def:source}
A \emph{source} $\sigma$ on a closed box~$X$ in the
coordinates~\eqref{eq:seven} consists of
\begin{enumerate}
\item a truncation degree $m\geq 5$ and a \emph{source window}~$W$, either the
full circle or a closed arc with endpoints that are rational multiples
of~$2\pi$;
\item a radial grid $0<r_0<r_1<\cdots<r_n<1$ of exact rationals;
\item for every $k$, a partition of~$W$ into finitely many arcs with
endpoints that are rational multiples of~$2\pi$, together with a non-negative
dual vector for the objective
$\Real\sum_{\nu=2}^{m}c_\nu r_k^{\nu}e^{i\nu\theta}$ at every endpoint
$\theta$ of that partition;
\item optionally, at each of those points, non-negative dual vectors for the
two second-derivative objectives of Section~\ref{sub:curv}.
\end{enumerate}
Given $\sigma$ and an image bin $I_j$, the estimates
\eqref{eq:arc}--\eqref{eq:radial} on the window
$I_j+[-\beta(r_{k+1}),\beta(r_{k+1})]$ assign to every annulus
$[r_k,r_{k+1}]$ a certified excess $d_{\sigma,k,j}\geq 0$ over the
Liu--Minda baseline, with $d_{\sigma,k,j}=0$ if that window is not contained
in $W$.  If $W$ is the full circle, the harmonic
estimate~\eqref{eq:harmonic} assigns to every annulus a further excess
$d^{\circ}_{\sigma,k}\geq0$, valid in every image direction, and
$\rho^{\circ}_\sigma=\tfrac12+\sum_kd^{\circ}_{\sigma,k}$ is the
\emph{origin radius} of $\sigma$; for a sectorial source put
$\rho^{\circ}_\sigma=\tfrac12$.  The vector
\[
  \Bigl(\max\Bigl(\rho^{\circ}_\sigma,\
   \tfrac12+\sum_{k=0}^{n-1}d_{\sigma,k,j}\Bigr)\Bigr)_{j}
\]
indexed by the image bins is the \emph{profile} of $\sigma$.  Below $r_0$
and beyond $r_n$ only the baseline is used.
\end{definition}

\begin{proposition}\label{prop:conditional}
Let $X$ be a closed box in the coordinates~\eqref{eq:seven} contained in the
normalized ranges~\eqref{eq:normalization}, let $\sigma_1,\ldots,\sigma_N$ be
sources on $X$, let $0 \leq \varepsilon \leq \tfrac12$, and assume the following.
\begin{enumerate}
\item[\textup{(H1)}] Every dual vector occurring in some $\sigma_i$ is
non-negative, and the bound~\eqref{eq:dual} that it certifies is evaluated on
the box $X\times\prod_{k=6}^{m}[-B_k,B_k]^2$ or on an outward-rounded box
containing it, where $m$ is the degree of $\sigma_i$: the
coordinates~\eqref{eq:seven} range over $X$ itself and every retained coordinate
beyond $c_5$ over its universal range~\eqref{eq:universal-bound}.  The
coefficients beyond the degree $m$ are not retained and enter only through the
tail~\eqref{eq:tail}.
\item[\textup{(H2)}] At every radius of every $\sigma_i$ the stored arcs are
contiguous and their union is the source window, and the
allowance~\eqref{eq:arc} is applied on each arc with an upper bound
$M_\theta\geq0$ for the angular second derivative, either the termwise
bound~\eqref{eq:termwise} or the certified
bound~\eqref{eq:certified-curvature}, the latter replaced by $0$ if it is
negative.
\item[\textup{(H3)}] For every annulus $[r_k,r_{k+1}]$ of every $\sigma_i$
and every image bin $I_j$, the excess $d_{\sigma_i,k,j}$ is obtained
from~\eqref{eq:radial} on the window $I_j+[-\beta(r_{k+1}),\beta(r_{k+1})]$
with a valid upper bound $M_r\geq0$ if that window lies in the source window
of $\sigma_i$, and is $0$ otherwise; if the source window is the full
circle, the excesses $d^{\circ}_{\sigma_i,k}$ are obtained
from~\eqref{eq:harmonic}.  In both estimates the tail~\eqref{eq:tail} beyond
the degree of $\sigma_i$ is subtracted.
\item[\textup{(H4)}] The function $\beta$ used to pass between image and
source angles satisfies~\eqref{eq:angle-distortion} on $X$, and $S_j \leq 1$
is an upper bound for $\sin\theta$ on $I_j$.
\end{enumerate}
For each image bin $I_j$, let $\rho_j$ be the largest of the profile values
at $j$ of the individual sources $\sigma_1,\ldots,\sigma_N$ and of the
annulus-wise combinations~\eqref{eq:annulus-mixing} formed over those subsets
$\mathcal S\subseteq\{\sigma_1,\ldots,\sigma_N\}$ whose radial grids coincide
exactly.  Then
\[
  \min_j\sqrt{\rho_j^2+\varepsilon^2-2\rho_j\varepsilon S_j}
\]
is a lower bound for $r(f)$ for every $f$ in the normalized class whose
coordinates~\eqref{eq:seven} lie in $X$.
\end{proposition}

\begin{proof}
Fix such an $f$.  Its coefficients satisfy the cuts $Ax \leq  b$ of
Section~\ref{sec:relax}, and by (H1) they lie in the box on which
each~\eqref{eq:dual} was evaluated.  Every stored dual vector therefore
certifies its objective at its stored point for this $f$.  By (H2) and
Lemma~\ref{lem:curvature} these pointwise bounds extend to every angle of the
source window at every stored radius, and by (H3) to every point of every
annulus of the grid on the window of the bin, the subtracted tail accounting
for the coefficients beyond the truncation degree.  Through
\eqref{eq:modulus-log} and~\eqref{lem:jensen}, each $d_{\sigma_i,k,j}$
is thus a valid excess for every point of the annulus whose source angle
lies in $I_j+[-\beta(r_{k+1}),\beta(r_{k+1})]$, and each
$d^{\circ}_{\sigma_i,k}$ for every point of the annulus.

Fix an image bin $I_j$ and consider the lift $\gamma_\theta$ of a ray with
$\theta\in I_j$.  By (H4) and Lemma~\ref{lem:angle}, every point $z$ of the
lift with $\abs{z} \leq  r_{k+1}$ has source angle in
$I_j+[-\beta(r_{k+1}),\beta(r_{k+1})]$, so the bounds of the previous
paragraph hold along the part of the lift in each annulus, and
Corollary~\ref{cor:excess} gives both
$R_\theta\geq \tfrac12+\sum_k d_{\sigma_i,k,j}$ and
$R_\theta\geq\rho^{\circ}_{\sigma_i}$ for each $i$ separately, while
Lemma~\ref{lem:mixing} gives the combined bound for each group with a common
grid.  A maximum of finitely many valid lower bounds is a valid lower bound,
so $R_\theta\geq \rho_j$ for every $\theta\in I_j$, and $\rho_j\geq\frac{1}{2}$
because all excesses are non-negative.  The displayed minimum is now the
bound of Section~\ref{sec:shift}, which Lemma~\ref{lem:star} identifies as
the radius of a schlicht disc.
\end{proof}

Hypothesis (H1) is what allows witnesses trained on different boxes and at
different degrees to be reused: each dual residual is recomputed on the final
box, and every retained coordinate outside~\eqref{eq:seven} is restored to
its universal interval.

\section{The finite coefficient cover}
\label{sec:cover}

The cover is a binary tree of boxes in the seven coordinates
\begin{equation}\label{eq:seven}
 (c_2,\Real c_3,\Real c_4,\Real c_5,
          \Imag c_3,\Imag c_4,\Imag c_5).
\end{equation}
Its root, the \emph{normalized root}, is the box
\[
  [0,1]\times[-B_3,B_3]\times[-B_4,B_4]\times[-B_5,B_5]
  \times[0,B_3]\times[-B_4,B_4]\times[-B_5,B_5],
\]
the ranges of Lemma~\ref{lem:logbody} for $c_2$ and $\Imag c_3$ and the
universal bounds~\eqref{eq:universal-bound} for the other five coordinates;
by that lemma it contains the coefficients of every normalized function.  Every
coefficient from $c_6$ onward is unrestricted apart from the universal
necessary conditions.  Each internal node is split in $c_2$, in
$v=\Imag c_3$ or, on a few leaves added last, in $p=\Real c_3$.  The tree has
$2357$ nodes in all: $1178$ internal ones and $1179$ leaves, with maximum
depth $35$.  Two kinds of contraction are applied to a node box before it is
used.  Lemma~\ref{lem:quadratic} contracts the range of $\Real c_3$ from the
lower endpoint of the current $c_2$ interval, at $696$ of the $2357$ nodes.
The second kind is certified by a dual vector.

\begin{lemma}[LP cap]\label{lem:cap}
Let $l \leq  x \leq  h$ be a box on which the cuts $Ax \leq  b$ hold, and let $j$
be a coordinate.  If $y\geq 0$ is any non-negative vector and $\lambda_j$
denotes the right side of~\eqref{eq:dual} for the objective $d=e_j$, then
$x_j\geq \lambda_j$ on the box.  Likewise, $y'\geq 0$ and $d=-e_j$ give
$x_j \leq -\lambda'_j$.  Consequently, every function whose coefficients lie
in the box has $x_j\in[\max(l_j,\lambda_j),\min(h_j,-\lambda'_j)]$, and the
box may be replaced by this contraction before any profile is evaluated.
\end{lemma}

\begin{proof}
This is Lemma~\ref{lem:dual} with the coordinate objectives $\pm e_j$.
\end{proof}

Contracted boxes are inherited by the children exactly like the quadratic
contraction.  A cap on $\Real c_3$ is carried at $1186$ of the $2357$ nodes.  On the box
$125/128 \leq  c_2 \leq 63/64$, $57/64 \leq \Imag c_3 \leq 115/128$, for example,
the cap gives $\abs{\Real c_3} \leq 0.5892$ where Lemma~\ref{lem:quadratic}
gives $\sqrt3/2$.  Every contraction of either kind is independently
checked.

Figure~\ref{fig:cover} shows the whole cover, and Table~\ref{tab:cover}
groups the leaves by the nine original intervals in $v$.  It reports the number of final leaves and the smallest 80-digit Arb
bound within each group.  Displayed bounds are truncated downward.

\begin{table}[ht]
\centering
\caption{The final coefficient cover, grouped by $\Imag c_3$.  The last
column is the smallest certified lower bound for the radius of a schlicht
disc over the leaves of the group.}\label{tab:cover}
\begin{tabular}{lrr}
\toprule
$\Imag c_3$ & number of leaves & smallest checked radius\\
\midrule
$[0,1/8]$   & 12  & $0.51003408679962$\\
$[1/8,1/4]$ & 59  & $0.51003053290060$\\
$[1/4,3/8]$ & 15  & $0.51003568733938$\\
$[3/8,1/2]$ & 7   & $0.51022166952328$\\
$[1/2,5/8]$ & 4   & $0.51039872699529$\\
$[5/8,3/4]$ & 9   & $0.51005217008693$\\
$[3/4,1]$   & 148 & $0.51044403023758$\\
$[1,2]$     & 923 & $0.51039389293769$\\
$[2,B_3]$   & 2   & $0.51152613239881$\\
\bottomrule
\end{tabular}
\end{table}

The smallest bound occurs on the box
\[
 \frac{15}{16} \leq  c_2 \leq 1,\qquad
 \frac{49}{256} \leq \Imag c_3 \leq \frac{99}{512}.
\]
On this box the LP cap of Lemma~\ref{lem:cap} gives
$\abs{\Real c_3} \leq 1.0451$ (rounded outward in the stored certificate),
where the quadratic contraction of Lemma~\ref{lem:quadratic} alone would
give $\abs{\Real c_3} \leq \sqrt2$; the other four coordinates in
\eqref{eq:seven} retain their universal ranges.  Table~\ref{tab:cover} shows
that the three groups with $\Imag c_3<3/8$ contain only $86$ of the $1179$
leaves but all of the smallest bounds, and within these groups the binding
boxes lie at the edge $c_2=1$.  This is where the argument is now weakest.
By contrast, the groups $3/4 \leq \Imag c_3 \leq 2$ hold $1071$ leaves, the
result of the subdivision that the earlier search concentrated there
(Section~\ref{sec:search}), and their smallest bounds are now about
$4\cdot10^{-4}$ above the target.

\input{figures/cover}

\begin{proof}[Proof of Theorem~\ref{thm:main}]
By Lemma~\ref{lem:logbody}, every normalized locally univalent Bloch
function has coefficients in the normalized root: the ranges
$0 \leq  c_2 \leq 1$, $0 \leq \Imag c_3 \leq  B_3$ and $\abs{\Real c_3} \leq  B_3$,
with $\Real c_4,\Imag c_4,\Real c_5,\Imag c_5$ and every later coefficient at
their universal ranges~\eqref{eq:universal-bound}.  The binary splits in the
machine-readable tree cover those ranges, and every contraction of a node box
follows from Lemma~\ref{lem:quadratic} or Lemma~\ref{lem:cap}.  Thus every
normalized function belongs to at least one leaf.

On every leaf, the hypotheses (H1)--(H4) of
Proposition~\ref{prop:conditional} hold for the sources selected there, with
the shift $\varepsilon$ recorded in the certificate; Section~\ref{sec:verify}
describes the independent verification of each of them.  The proposition then
gives the radius represented in Table~\ref{tab:cover}.  The minimum over the
leaves is strictly larger than $0.51$.  Hence $\Binf>0.51$.
Equation~\eqref{eq:relation} gives the same lower bound for $L$.
\end{proof}

\section{Rigorous verification}
\label{sec:verify}

The proof package separates discovery from checking.  NumPy and SciPy/HiGHS
are used during search.  The verifier does not call an LP solver and does not
trust stored primal values, objective values, angular minima, or diagnostic
radii.  Its accepted inputs consist of exact finite data from which the
inequalities are reconstructed.

The main manifest is the file
\begin{center}
\texttt{landau-cover-split-051-caps.json}
\end{center}
of the archive published at \url{https://doi.org/10.5281/zenodo.22849858}.
It contains a complete binary tree, source filenames and SHA-256 digests,
exact split points, LP caps with their dual vectors and endpoints, source
indices for each leaf, exact decimal shifts, and the global target.  The
source files contain the cut descriptors, radial grids, angular arc trees,
sparse non-negative dual multipliers, and the curvature duals of every stored
point.

The manifest contains $775$ witness sources, all at degree $24$ and all with
curvature duals, in addition to the root source, so $776$ in all.  One source
is sectorial and the others cover full circles.  A leaf uses between $2$ and
$37$ sources, and every directional profile has $256$ image-angle bins.  For
each leaf, the checker reconstructs every selected source on that leaf box,
groups identical grids for~\eqref{eq:annulus-mixing}, and takes the pointwise
maximum with every original source profile.  Nearly equal floating-point
radial values remain distinct exact rationals, and their annular
contributions are never mixed.

The checker performs the following tasks.

\begin{enumerate}
\item It verifies that the root contains the normalized ranges for every
coefficient and that each binary split has two children meeting at the same
exact dyadic point.  The endpoint used for $B_3$ in the machine-readable root is a binary64
number, $554940230307281/2^{47}$, whose shortest decimal representation is
$3.943087494258755$; the verifier checks that it exceeds the analytic bound
$B_3=3.94308749425867\ldots$ of~\eqref{eq:universal-bound}.
The endpoints of adjacent leaves are included in both children, so no
boundary case is omitted.
\item It checks the premises of every quadratic contraction and recomputes
the residual bound~\eqref{eq:dual} of every LP cap on the node box.  A cap is
stored as the two dual vectors, one source whose cut pool they refer to, and
the two claimed endpoints; an endpoint is accepted only if it lies outside
the certified interval, and a cap or an empty contracted box without a proof
is rejected.
\item It verifies each source digest, reconstructs all moment, tangent, and
point cuts in arbitrary-precision interval arithmetic, and evaluates every
dual residual using~\eqref{eq:dual} on the leaf box, with the universal range
of~\eqref{eq:universal-bound} for every retained coordinate
outside~\eqref{eq:seven}.  This discharges~(H1) of
Proposition~\ref{prop:conditional}.
\item It verifies exact angular coverage at every stored radius, evaluates
the curvature duals of Section~\ref{sub:curv} where present, and applies
the Taylor allowance~\eqref{eq:arc} with the smaller of the termwise and the
certified second-derivative bound, the latter replaced by zero if it is
negative.  This discharges~(H2).
\item It covers every radial gap using~\eqref{eq:radial} on the window of
each image bin and, for full-circle sources, also~\eqref{eq:harmonic},
includes the explicit tail~\eqref{eq:tail}, and computes every
image-direction profile together with the origin radius.  On identical
radial grids it forms the certified
annulus-wise maxima~\eqref{eq:annulus-mixing}, retaining all original
profiles as fallbacks, and then computes every shifted distance from
$\beta$ and the bins, discharging~(H3) and~(H4).
\item It compares every leaf with its local target and compares the minimum
of all leaves with the global target only after the whole tree is closed.
\end{enumerate}

The verification was run at $80$ digits in twelve independent shards, each
writing the enclosure of every leaf it checked; the twelve shards took a
little over two hours of wall time on an Apple M4~Pro, and a single leaf
takes a few minutes on a laptop.  A separate aggregation step then accepts
the shard logs only if every leaf of the manifest appears in exactly one of
them and exceeds the target, so the split cannot omit a leaf or count one
twice; the whole manifest can equally be checked in a single process.  The
archived shard logs are bound to the manifest and to the checker by a
provenance record listing the SHA-256 digests of all three.  The minimum
over the $1179$ leaves is the enclosure stated in
Theorem~\ref{thm:main}, and the checked margin over the target $0.51$ is
approximately $3.05\cdot10^{-5}$, vastly larger than the width of that
enclosure.

Adversarial tests reject inexact and endpoint splits, unjustified
contractions and caps, malformed or negative curvature duals, missing
sources, changed source hashes, hidden high-coefficient restrictions, open
branches, inconsistent annular grids, and excessive local targets.

Every binary64 number in a source is interpreted as the exact rational it
represents.  Transcendental quantities such as logarithms, exponentials,
trigonometric functions, and $\pi$ are evaluated directly by Arb with outward
rounding.  The commands that reproduce the verification, the twelve shard
logs and the aggregation output are part of the published archive described
in the data availability statement below.

\section{Search strategy and present limitations}\label{sec:search}

The progression from $1/2$ to the present result was not obtained by blind
subdivision.  A global moment profile first gave $0.5023$; certified
pointwise cuts and directional image bins raised this to $0.5035$, a cover
in $\Imag c_3$ reached $0.507$, and an adaptive search that subdivided only
in $c_2$ and $\Imag c_3$, retrained dual witnesses on weak boxes, combined
profiles annulus-wise as in~\eqref{eq:annulus-mixing}, and shared productive
witnesses between boxes closed a cover of $1095$ leaves at $0.508$.  The
resumable controller that drives this search accompanies the proof package;
its output is diagnostic only, and a numerical candidate becomes part of the
theorem only after it has been assembled into a complete exact cover and
accepted by the verifier of Section~\ref{sec:verify}.

A floating-point mirror of the checker then located the loss on the weakest
leaves of that cover. With the same witnesses, setting the infinite tail or the
angular distortion $\beta$ to zero changed the certified radius by at most
$9\cdot10^{-4}$, whereas removing the angular Taylor allowance moved it from
$0.5080$ to $0.5165$.

Certifying the second derivatives (Section~\ref{sub:curv}), refining the arcs
in the binding source directions, and adding LP caps (Lemma~\ref{lem:cap})
lifted the weakest leaf to $0.5123$ and the whole $1095$-leaf cover to a
certified $0.5088$ without training a single new source: every new
multiplier, whether for a second-derivative objective, a refined arc endpoint
or a coordinate cap, was found over the cut pool of an existing source.  The
$14$ leaves then below $0.51$ were wide, shallow slabs in $\Imag c_3<3/4$ that
had never been refined because they cleared $0.508$ easily.  One freshly trained
witness per slab, followed by subdivision of the six widest slabs in $c_2$,
$\Imag c_3$ and $\Real c_3$ with the new witnesses reused, produced the final
$1179$-leaf cover.

The present limitations are quantitative.  A new witness trained with
curvature duals on one of the wide slabs costs about an hour of compute and
raises its bound by between $2\cdot10^{-5}$ and $5\cdot10^{-4}$; the slabs
$\Imag c_3<3/8$ near $c_2=1$ sit almost uniformly between $0.5100$ and
$0.5102$, so that further subdivision without new witnesses multiplies the
leaves without moving the minimum.  Two features of the method compound
this.  The relaxation keeps finitely many cuts on a few low coefficients, so
a minimizer of a relaxed program need not be the coefficient sequence of any
function in the class; and the bound of Proposition~\ref{prop:conditional}
is a minimum over the box taken separately in every image direction.  The
second effect was visible on the boxes near $c_2=1$ with
$3/4 \leq \Imag c_3 \leq 2$, the bottleneck of the earlier covers.  In the
normalization~\eqref{eq:normalization} the Liu--Minda extremal function has
$c_2=1$ and $c_3=4i/3$, and its weak direction is the source angle
$270^\circ$; on those boxes the relaxation also admits coefficient vectors
whose weak direction lies near $100^\circ$, and the profile is weak in both
directions although no single function need be.  Numerical experiments with
the present profiles stalled near $0.52$.  A stall of this kind does not
establish a ceiling of the method, and we do not know how much of the
remaining gap to~\eqref{eq:upper} is due to the relaxation and how much to
the direction-wise minimum.  No claim above $0.51$ is made here.

\subsection*{Data and code availability}

The certificate archive is published at Zenodo,
\url{https://doi.org/10.5281/zenodo.22849858}.  It contains the manifest, the
$776$ referenced certificate sources ($3.2$~GB), and all code needed to run the
full verification. It takes half a minute to aggregate the shard logs, about
four minutes to verify a single leaf on a laptop, and about two hours to rerun
the complete verification on twelve cores.  The archive also contains the search
scripts that produced the manifest from the stored sources, with a note on how
they were run.  Earlier certificates at $0.508$ and $0.5088$ are available
from the author.

\subsection*{Tool disclosure}

OpenAI Codex (GPT-6 Astra) and Anthropic Claude Code (Fable 5.1), both accessed
September 2026, were used as interactive research assistants for mathematical
brainstorming, numerical exploration, proof auditing, and development of the
interval-arithmetic verification scripts. The author has reviewed the generated
material and takes responsibility for the arguments and computations.


\begin{thebibliography}{99}

\bibitem{Arb}
F.~Johansson,
\emph{Arb: efficient arbitrary-precision midpoint-radius interval arithmetic},
IEEE Trans. Comput. \textbf{66} (2017), 1281--1292.
\url{https://doi.org/10.1109/TC.2017.2690633}

\bibitem{Ahlfors}
L.~V.~Ahlfors,
\emph{An extension of Schwarz's lemma},
Trans. Amer. Math. Soc. \textbf{43} (1938), 359--364.
\url{https://doi.org/10.1090/S0002-9947-1938-1501949-2}

\bibitem{AhlforsGrunsky}
L.~V.~Ahlfors and H.~Grunsky,
\emph{\"Uber die Blochsche Konstante},
Math. Z. \textbf{42} (1937), 671--673.
\url{https://doi.org/10.1007/BF01160101}

\bibitem{ChenShiba}
H.~Chen and M.~Shiba,
\emph{On the locally univalent Bloch constant},
J. Anal. Math. \textbf{94} (2004), 159--170.
\url{https://doi.org/10.1007/BF02789045}

\bibitem{Landau}
E.~Landau,
\emph{\"Uber die Blochsche Konstante und zwei verwandte Weltkonstanten},
Math. Z. \textbf{30} (1929), 608--634.
\url{https://doi.org/10.1007/BF01187791}

\bibitem{LiuMinda}
X.~Liu and D.~Minda,
\emph{Distortion theorems for Bloch functions},
Trans. Amer. Math. Soc. \textbf{333} (1992), 325--338.
\url{https://doi.org/10.1090/S0002-9947-1992-1055809-0}

\bibitem{Minda}
C.~D.~Minda,
\emph{Bloch constants},
J. Anal. Math. \textbf{41} (1982), 54--84.
\url{https://doi.org/10.1007/BF02803394}

\bibitem{Peschl}
E.~Peschl,
\emph{\"Uber die Verwendung von Differentialinvarianten bei gewissen
Funktionenfamilien und die \"Ubertragung einer darauf gegr\"undeten Methode
auf partielle Differentialgleichungen vom elliptischen Typus},
Ann. Acad. Sci. Fenn. Ser. A I \textbf{336/6} (1963), 23 pp.

\bibitem{Pommerenke}
Ch.~Pommerenke,
\emph{On Bloch functions},
J. London Math. Soc. (2) \textbf{2} (1970), 689--695.
\url{https://doi.org/10.1112/jlms/2.Part_4.689}

\bibitem{Rademacher}
H.~Rademacher,
\emph{On the Bloch--Landau constant},
Amer. J. Math. \textbf{65} (1943), 387--390.
\url{https://doi.org/10.2307/2371689}

\bibitem{Rettinger}
R.~Rettinger,
\emph{On computable approximations of Landau's constant},
Log. Methods Comput. Sci. \textbf{8} (2012), no.~4, paper~15.
\url{https://doi.org/10.2168/LMCS-8(4:15)2012}

\bibitem{Wikstrom}
F.~Wikstr\"om,
\emph{An improved lower bound for Bloch's constant},
preprint (2026). \url{https://arxiv.org/abs/2608.17660}

\bibitem{Yanagihara}
H.~Yanagihara,
\emph{On the locally univalent Bloch constant},
J. Anal. Math. \textbf{65} (1995), 1--17.

\end{thebibliography}
\end{document}